\documentclass[letterpaper, 10 pt, conference]{ieeeconf}  

\IEEEoverridecommandlockouts 
\usepackage{graphicx} 
\usepackage{amsmath} 
\usepackage{amssymb}  
\usepackage{subfig}
\usepackage[english]{babel}
\usepackage{xcolor}

\newcommand{\eome}{{\rm \nabla\kern-.6em \nabla}}
\newcommand{\diag}{\mathrm{diag}}
\newcommand{\Argmax}{\mathrm{Argmax}}
\newtheorem{definition}{Definition}
\newtheorem{theo}{Theorem}

\newtheorem{corollary}{Corollary}
\newtheorem{remarkTh}{Remark}
\newtheorem{proposition}{Proposition}

\newenvironment{remark}[0]{\begin{remarkTh}}{\hfill$\square$\end{remarkTh}}

\usepackage[bb=boondox]{mathalfa}

\newcommand{\infeq}{\leqslant}
\newcommand{\V}{\mathcal{V}}

\newcommand{\R}{\mathbb{R}}
\renewcommand{\S}{\mathbb{S}}

\renewenvironment{proof}[1][Proof]{\textbf{#1.} }{\ \rule{0.5em}{0.5em}}

\renewcommand\epsilon{\varepsilon}

\title{\LARGE \bf Stabilization of an unstable wave equation using an infinite dimensional dynamic controller}

\author{Matthieu~Barreau, Fr\'ed\'eric~Gouaisbaut and Alexandre~Seuret
\thanks{M. Barreau, F. Gouaisbaut, A. Seuret are with LAAS - CNRS, Universit\'e de Toulouse, CNRS, UPS, France e-mail: (mbarreau,aseuret,fgouaisb@laas.fr).}
\thanks{This work is supported by the ANR project SCIDiS contract number 15-CE23-0014.}
}

\begin{document}

\maketitle
\thispagestyle{empty}
\pagestyle{empty}

\begin{abstract} This paper deals with the stabilization of an anti-stable string equation with Dirichlet actuation where the instability appears because of the uncontrolled boundary condition. Then, infinitely many unstable poles are generated and an infinite dimensional control law is therefore proposed to exponentially stabilize  the system. The idea behind the choice of the controller is to extend the domain of the PDE so that the anti-damping term is compensated by a damping at the other boundary condition. Additionally, notice that  the system can then be exponentially stabilized with a chosen decay-rate and is robust to uncertainties on the wave speed and the anti-damped coefficient of the wave equation, with the only use of a point-wise boundary measurement. The efficiency of this new control strategy is then compared to the backstepping approach. 
\end{abstract}

\section{Introduction}

In recent years, we have seen a renewed interest in the control of infinite dimensional systems for both practical and theoretical considerations. Indeed, many complex systems may be easily modeled by Partial Differential Equations (PDE).  These include delay systems \cite{SAFI20171}, string/payload \cite{he2014adaptive}, MEMS \cite{flores2014dynamics} or drilling pipes \cite{bresch2014output,marquez2015analysis}, among many others \cite{bastin2016stability}. From a theoretical point of view, one has witnessed many contributions to these problems: backstepping method \cite{krstic2008boundary}, saturated control \cite{prieur2016wave} or event-based control~\cite{espitia2016event}.

Concerning the case of hyperbolic PDE and especially the string equation in a finite domain, even if the model is quite simple, there exist various control laws which can be distributed or only at the boundary for example. Firstly, notice that many different kinds of instabilities can affect the system as for instance internal anti-damping \cite{FREITAS1996338,ASJC:ASJC1509} or an unstable Robin boundary condition \cite{KRSTIC200863}. These two instabilities generally lead to a finite number of unstable poles. Another possible boundary condition which induces infinitely many unstable poles are reported in \cite{Lagnese1983163}. As noted in \cite{Lagnese1983163,lasiecka1992uniform}, this instability arises from the unstable difference operator which appears if the wave equation is modeled as a neutral time-delay system. The control of this anti-stable wave is therefore much more challenging. Furthermore, in general, if a standard feedback control law is designed, it is known to be not robust to input/output-delays \cite{datko1988not,datko1986example}. 

This paper deals with the stabilization of an antistable wave equation of the latter kind with Dirichlet actuation. Several methodologies have been proposed to stabilize this model as \cite{gugat2015exponential,lasiecka1992uniform}. Among them, a very popular approach refers to the backstepping methodology for infinite-dimensional systems introduced in \cite{krstic2009delay,krstic2008boundary}. The idea is to determine a feedback law such that the closed-loop system behaves as an asymptotically or exponentially stable  system with the desired properties, for instance a one side boundary damped wave equation. This leads to the design of an infinite dimensional control law which requires a distributed measure all over the domain \cite{krstic2009adaptive,smyshlyaev2009boundary}. Notice that if these measurements are not available, an observer can be designed to estimate this whole state, measuring the state and  its derivative at one boundary \cite{KRSTIC200863}.
The proposed methodology follows the same starting point. We aim at finding a control law in order to get, in closed-loop, a two sided boundary damped wave equation but contrary to the backstepping approach the target system is extended in the space domain (see for instance  \cite{hansen1995exact}).  This new idea simplifies the design of the control law. Firstly, it results in a very simple dynamic control law of infinite dimension. Secondly, the control law requires the measurement of only the state at one boundary. Furthermore, this methodology allows to obtain performances similar to approaches such as backstepping. An interesting feature relies also on the robustness of the approach since the parameters of the original system could be uncertain. 
 
The paper is organized as follows. In Section 2, the model for the wave equation as well as the control objectives are detailed. In Section~3, a new controller design is proposed. The existence and uniqueness of a solution to the closed-loop system is then studied. Therefore, an exponential stability result is derived and some extensions are provided. As the main result is formulated in terms of a Linear Matrix Inequality (LMI), the exponential stability result is extended in Section~4 to a robust stability analysis. 
Section 5 provides a numerical application of the proposed controller on two examples together with a comparison with the backstepping methodology.

\noindent \textbf{Notations:} Throughout this paper, the notation $u_t$ stands for $\frac{\partial u}{\partial t}$. The common spaces of square integrable functions on $[0,1]$ is denoted $L^2 = L^2([0,1]; \mathbb{R})$ and $H^n=\left\{ z \in L^2; \forall m \infeq n, \frac{\partial^m z}{\partial x^m} \in L^2 \right\}$ for the Sobolov spaces. 
$L^2$ is equipped with the norm $\|z\|^2 =  \int_0^1 |z(x)|^2  dx = \left<z,z\right>$. 
For any square matrices $A$ and $B$, $\text{diag}$ is defined as $\text{diag}(A,B) = \left[ \begin{smallmatrix}A & 0\\ 0 & B \end{smallmatrix} \right]$.
A matrix $P \in \R^{n \times n}$ is positive definite if it belongs to the set $\S^n_+$ or more simply $P \succ 0$. $I_{n}$ is the identity matrix of dimension $n\times n$ and $0_{n,m}$ is the null matrix of size $n \times m$.
\section{Problem Statement}
The wave equation studied in this paper is described by the following model:
\begin{equation} \label{eq:system1}
	\left\{
	\begin{array}{ll}
		u_{tt}(x,t) = c_1^2 u_{xx}(x,t), \quad & x \in (0, 1), \\
		u_x(0,t) = g u_t(0,t), & \\
		u(1,t) = w(t), \\
		y(t) = u_x(1,t), \\
		u(x,0) = u^0(x), \ u_t(x, 0) = u_t^0(x), & x \in (0, 1),
	\end{array}
	\right.
\end{equation}
 
 It represents the evolution of a wave equation of amplitude $u$ of speed $c_1$. At $x = 0$, there is a well-known boundary damping condition \cite{Lagnese1983163}. At $x = 1$, there is a Dirichlet actuation and $w$ is the control law. The only measurement is $y$, which is the space derivative of $u$ at $x = 1$. 

$g$ is closely related to the reflection coefficient at the boundary $x = 0$. It is well-known from \cite{Lagnese1983163} that for $g < 0$, this wave is unstable. Actually, this issue has also been discussed in \cite{hale2001effects} where it is compared to a neutral time-delay system. Indeed, for $g < 0$, the neutral time-delay system has a non-stable difference operator, making its stabilization possible only if $u_t(1, \cdot)$ is exactly and perfectly measured. This measurement is, in general, difficult to get since it relies on specific sensors which cannot provide the derivative at time $t$ but at a sightly delayed time. As enlighten in \cite{datko1988not,datko1986example}, a proportional feedback control on $u_t(1,\cdot)$ is not robust to time-delay. That is the reason why we need to consider that the measure of $u_t(1,\cdot)$ is not available, making the synthesis of a control law a real challenge. Notice that the famous methodology of backstepping, described in \cite{krstic2009delay} explains for instance how to build such a control law but it relies on the full distributed measurements of the state $u$ which are practically difficult to get.


Here, considering $g < 0$, we aim at showing that there exists an infinite-dimensional controller ensuring which is not issued from a backstepping methodology ensuring the $L_2$-stability of the closed-loop system without an explicit measurement of $u_t(1, \cdot)$.

\section{Controller Design}

The proposed controller is as follows where $h> 0$ and $q$ are the control design parameters:
\begin{equation} \label{eq:controller1}
	\left\{
	\begin{array}{ll}
		v_{tt}(x,t) = c_2^2 v_{xx}(x,t), \quad & x \in (0, 1), \\
		v_x(0,t) = y(t), & \\
		v_x(1,t) = -h v_t(1,t) - qv(1,t), \\
		w(t) = v(0,t) + r(t), \\
		v(x,0) = v^0(x), \ v_t(x, 0) = v_t^0(x), & x \in (0, 1).
	\end{array}
	\right.
\end{equation}

The reference is $r$ and we assume, without loss of generality, that $r \equiv 0$. This control is of infinite dimension and describes a wave equation. Even if an explicit control formulation of $w$ depending on the initial conditions and $y$ can be derived (using \cite{louw2012forced} for example), this controller is seen as an infinite-dimensional dynamic controller depicted in Fig.~\ref{fig:closedLoop1}. Note that, if $c_1 = c_2 = c$, controller~\eqref{eq:controller1} is an extension of system~\eqref{eq:system1}. Indeed, the closed-loop system is then an extended wave over a domain of length $2$ and speed $c$ with two damping or anti-damping boundary conditions. The stability of this interconnected system seems then quite simple and a stability test is provided in Section~3. The most interesting part comes when deriving a robust stability criterion to uncertainties on $c_1$ and $g$, as discussed in Section~4.

\begin{remark} The same methodology applies (and similar results are obtained) considering other boundary conditions: $u_x(1,t) = w(t)$ and $y(t) = u(1,t)$ for system~\eqref{eq:system1} while for the controller we use: $v(0,t) = y(t)$. The closed-loop is still an extension of the wave on the larger interval $(0, 2)$. \end{remark}

\begin{figure}
	\centering
	\includegraphics[width=8.4cm]{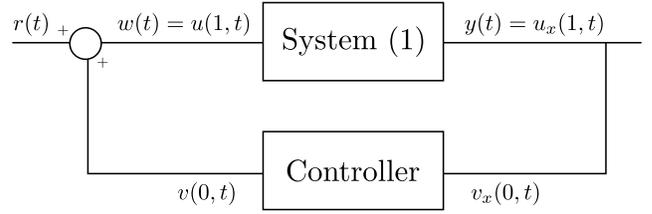}
	\caption{Block diagram of closed-loop system \eqref{eq:system1}-\eqref{eq:controller1}.}
	\label{fig:closedLoop1}
\end{figure}

\subsection{Existence and uniqueness}

%

As the string equation is a second order in time and space, the state $u$ must be regular enough such that the derivations have a sense. Then, as done in \cite{krstic2009delay,tucsnak2009observation}, the following space is defined:
\[
	\mathcal{H} = \left\{ (u, u_t, v, v_t) \in H^1 \times L^2 \times H^1 \times L^2, \ u(1) = v(0) \right\}.
\]

In some practical cases like in drilling systems \cite{bresch2014output}, a convergence in speed is required with no care on the position. Then, similarly to \cite{krstic2009delay,smyshlyaev2009boundary}, a seminorm on $\mathcal{H}$ is defined:
\begin{equation} \label{eq:seminorm}
	\|(u, u_t, v, v_t)\|_{\mathcal{H}}^2 = c_1^2 \|u_x\|^2 + \|u_t\|^2 + c_2^2 \|v_x\|^2 + \|v_t\|^2.
\end{equation}
Together with space $\mathcal{H}$, this is a semi-norm because a convergence in the sense of $\| \cdot \|_{\mathcal{H}}$ means $u_t, u_x, v_t$ and $v_x$ are converging to $0$ but there is no constraint on $u$ and $v$.

If a convergence in position is needed, the previous subspace is equipped with the following norm:
\begin{equation} \label{eq:norm}
	\|(u, u_t, v, v_t)\|_{\mathcal{H}_0}^2 = \|(u, u_t, v, v_t)\|_{\mathcal{H}}^2 + v(1)^2.
\end{equation}
This semi-norm implies the convergence of $u_x$, $u_t$, $v_x$, $v_t$ and $v(1)$ meaning $u$ and $v$ are converging to $0$ in $L_2$ norm.
With these previous definitions, $(\mathcal{H}, \|\cdot\|_{\mathcal{H}})$ and $(\mathcal{H}, \|\cdot\|_{\mathcal{H}_0})$ are Hilbert spaces. 

\begin{definition} System~\eqref{eq:system1}-\eqref{eq:controller1} is said to be \textbf{dissipative} in $(\mathcal{H}, \|\cdot\|_{\mathcal{H}})$ (resp. $(\mathcal{H}_0, \|\cdot\|_{\mathcal{H}_0})$) if there is a seminorm $\| \cdot \|$ equivalent to $\|\cdot\|_{\mathcal{H}}$ (resp. $\|\cdot\|_{\mathcal{H}_0}$) for which $\frac{d}{dt} \| (u, u_t, v, v_t) \| < 0$.
\end{definition}

The following proposition states the existence and uniqueness of solutions to \eqref{eq:system1}-\eqref{eq:controller1}.

\begin{proposition} For any initial condition $(u^0, u_t^0, v^0, v_t^0) \in \mathcal{H}$, there exists a unique solution to the latter system if system \eqref{eq:system1}-\eqref{eq:controller1} is dissipative in $(\mathcal{H}, \|\cdot\|_{\mathcal{H}})$.
\end{proposition}

The proof is very similar to the one given in \cite[Ch. 3.9]{tucsnak2009observation} on the wave equation with boundary damping.  Assuming the dissipativity of the abstract operator $T$ related to system \eqref{eq:system1}-\eqref{eq:controller1}, it is enough to show that $T$ is invertible to apply Lumer-Phillips Theorem (see for instance Theorem~3.8.4 in \cite{tucsnak2009observation}). Once the solution is defined, the study of its equilibrium points can be pursued.

\begin{proposition} \label{sec:equilibirum} An equilibrium point $(u^e, v^e)$ of system \eqref{eq:system1}-\eqref{eq:controller1} verifies: $u^e = v^e \in \mathbb{R}$ and $q u^e = 0$. If $q \neq 0$, then the only equilibrium point is $0_{\mathcal{H}}$.
\end{proposition}
\begin{proof} Assume $(u^e, v^e)$ is an equilibrium point. Then for $x \in (0, 1)$, $u^e$ and $v^e$ are two first order polynomials. The boundary condition on $v^e$ at $x = 1$ gives $v^e(0) = 0$ if $q\neq 0$. Since $u^e_x(0) = 0$, we get $u^e \in \mathbb{R}$. Then, since we have $u^e(1) = v^e(0)$, that gives $u^e = 0$ and, consequently, $u^e = 0$. If $q = 0$, we only get $u^e = v^e \in \mathbb{R}$. \end{proof}

The desired property throughout this paper is the exponential stability of system \eqref{eq:system1}-\eqref{eq:controller1}  defined as follows.
\begin{definition} \label{def:expo} A solution of system \eqref{eq:system1}-\eqref{eq:controller1} with initial condition $(u^0, u_t^0, v^0, v_t^0) \in \mathcal{H}$ is said to be \textbf{$\mathcal{H}$-exponentially stable} if there exist $\gamma > 1, \alpha > 0$ such that the following inequality holds for $t \geq 0$:
\begin{equation} \label{eq:expEstimate}
	\small{\|(u(t), u_t(t), v(t), v_t(t))\|_{\mathcal{H}} \leq \gamma \|(u^0, u_t^0, v^0, v_t^0) \|_{\mathcal{H}} e^{-\alpha t}}.
\end{equation}
\end{definition}
\vspace{0.1cm}
\begin{definition} \label{def:expo2} A solution of system \eqref{eq:system1}-\eqref{eq:controller1} with initial condition $(u^0, u_t^0, v^0, v_t^0) \in \mathcal{H}$ is said to be \textbf{$\mathcal{H}_0$-exponentially stable} if there exist $\gamma > 1, \alpha > 0$ such that the following inequality holds for $t \geq 0$:
\begin{equation} \label{eq:expEstimate2}
	\small{\|(u(t), u_t(t), v(t), v_t(t))\|_{\mathcal{H}_0} \leq \gamma \|(u^0, u_t^0, v^0, v_t^0) \|_{\mathcal{H}_0} e^{-\alpha t}}.
\end{equation}
\end{definition}
\vspace{0.1cm}

After these very general results, the proof of dissipativity is given thanks to the construction of a Lyapunov functional, which has a strictly negative derivative along the trajectories of system \eqref{eq:system1}-\eqref{eq:controller1}.

\subsection{Main theorem}

As a first step, we consider the case $q \neq 0$, and the following theorem is derived for a convergence in speed and position.

\begin{theo} \label{sec:theoUndamped} The unique solution of system \eqref{eq:system1}-\eqref{eq:controller1} with initial condition $(u^0, u_t^0, v^0, v_t^0) \in \mathcal{H}$ is $\mathcal{H}_0$-exponentially stable and converges to $0_{\mathcal{H}}$ if there exists real numbers $S_1, S_2, S_3, S_4, S_5 > 0$ such that the following LMI holds with $q \neq 0$:
\begin{equation} \label{eq:PsiNegative}
	\Psi_{c_1, c_1 g}(0) \prec 0,
\end{equation}
with 
\[
	\hspace{-0.2cm}
	\begin{array}{cl}
		\Psi_{c_1, c_1g}(\alpha) & \!\!\!\! = H_{c_1, c_1 g}^{\top} E_{\alpha}(1)S H_{c_1, c_1 g} - G_{c_1, g}^{\top} S G_{c_1, c_1 g} + Q_{\alpha}, \\
		H_{c_1, c_1 g} & \!\!\!\! = \left[ \begin{smallmatrix} 0 & c_1 & 1 & 0 & 0 \\ 1-c_1g & 0 & 0 & 0 & 0 \\ 0 & 0 & 0 & 1-c_2 h & -c_2 q \\ 0 & -c_2 & 1 & 0 & 0 \end{smallmatrix} \right], \\
		G_{c_1, c_1 g} & \!\!\!\! = \left[ \begin{smallmatrix} 1+c_1 g & 0 & 0 & 0 & 0 \\ 0 & -c_1 & 1 & 0 & 0 \\  0 & c_2 & 1 & 0 & 0 \\ 0 & 0 & 0 & 1+c_2h & c_2 q \end{smallmatrix} \right], \\
		Q_\alpha & \!\!\!\! = \diag \left(0_{3,3}, \left[ \begin{smallmatrix} 0 & S_5 \\  S_5 & 2 \alpha S_5 \end{smallmatrix} \right] \right), \\ 
		S & \!\!\!\! = \diag(S_1, S_2, S_3, S_4), \\
        E_{\alpha}(x) & \!\!\!\! = \diag(e^{2 \alpha x c_1^{-1}} I_2, e^{2 \alpha x c_2^{-1}}I_2).
	\end{array}
\]
\end{theo}

\begin{proof} Let us introduce the following variable: 
\[
	\chi(x,t) = \left[ \begin{matrix} u_t(x, t) + c_1 u_x(x, t) \\ 
							u_t(1-x, t) - c_1 u_x(1-x, t) \\
							v_t(x, t) + c_2 v_x(x, t) \\
							v_t(1-x, t) - c_2 v_x(1-x, t) 
			\end{matrix} \right],
\]
 for $t \geq 0, x \in (0, 1)$. This variable is based on modified Riemann invariant \cite{besselString,bastin2016stability}, which has the following property: $\chi_t = \Lambda \chi_x$ with $\Lambda = \diag(c_1, c_1, c_2, c_2)$. Following \cite{besselString,4060979,morgul1994dynamic}, we introduce a Lyapunov functional:
 \begin{equation} \label{eq:V1}
 	\V_{\alpha}(\chi, v(1)) = \int_0^1 \chi^{\top}(x) E_{\alpha}(x) \Lambda^{-1} S \chi(x) dx + S_5 v^2(1),
 \end{equation}
 where the time variable has been omitted for the sake of simplicity.
Note that $\V_{\alpha}$ is equivalent to $\| \cdot \|_{\mathcal{H}_0}$ and its derivative along the trajectories of system \eqref{eq:system1}-\eqref{eq:controller1} gives:
 \begin{equation} \label{eq:V1dot}
 	\begin{array}{cl}
 		\!\!\!\dot{\V}_{\alpha}(\chi, v(1)) & \!\!\!\! = 2 \displaystyle\int_0^1 \chi_x^{\top}(x) E_{\alpha}(x)S \chi(x) dx + 2 S_5 v(1) v_t(1) \\
 		& \!\!\!\! = \left[ \chi^{\top}(x) E_{\alpha}(x) S \chi(x) \right]_0^1 \\
 		& \hfill - 2 \alpha \displaystyle\int_0^1 \chi^{\top}(x) E_{\alpha}(x) \Lambda^{-1} S \chi(x) dx \\
 		& \hfill + 2 S_5 v(1) v_t(1) \\
 		& \!\!\!\!= -2 \alpha \V_{\alpha}(\chi, v(1)) +  \chi^{\top}(1) E_{\alpha}(1) S \chi(1) \\
 		& \hfill - \chi^{\top}(0) S \chi(0) + 2 \alpha S_5 v^2(1) + 2 S_5 v(1) v_t(1). \\
	\end{array}
 \end{equation}
 
 Introducing $\xi = \left[ u_t(0) \ \ u_x(1) \ \ v_t(0) \ \ v_t(1) \ \ v(1) \right]^{\top}$, the two states $\chi(0)$ and $\chi(1)$ can be rewritten as $\chi(0) = G_{c_1, c_1 g} \xi$, $\chi(1) = H_{c_1, c_1 g} \xi$ so that we get:
 \[
 	\dot{\V}_{\alpha}(\chi, v(1)) + 2 \alpha \V_{\alpha}(\chi, v(1)) = \xi^{\top} \Psi_{c_1, c_1g}(\alpha) \xi \infeq 0.
 \]
 
As it is continuous with respect to $\alpha$, if $\Psi_{c_1, c_1g}(0) \prec 0$, it is also the case for a sufficiently small $\alpha$. Then, the solutions converge exponentially with respect to $\| \cdot \|_{\mathcal{H}_0}$. %
\end{proof}




It is possible to prove that the previous theorem does not hold if $h < -1$ but the exponential convergence in semi-norm~\eqref{eq:seminorm} still holds, meaning that the solutions do not converge to $0_{\mathcal{H}}$ but the wave speeds $u_x$ and $u_t$ are indeed going exponentially to $0$. This weaker stability condition is expressed in the following corollary, dealing with the case where $q = 0$.

\begin{corollary} The unique solution of system \eqref{eq:system1}-\eqref{eq:controller1} with $q = 0$ and initial condition $(u^0, u_t^0, v^0, v_t^0) \in \mathcal{H}$ is $\mathcal{H}$-exponentially stable if there exists $S_1, S_2, S_3, S_4 > 0$ such that the following LMI holds:
\begin{equation} \label{eq:PsiTilde}
	\tilde{\Psi}_{c_1, c_1 g} = \begin{bmatrix}I_{4} &0_{4,1}\end{bmatrix} \Psi_{c_1, c_1 g}(0) \begin{bmatrix}I_{4} &0_{4,1}\end{bmatrix}^\top \prec 0,
\end{equation}
\end{corollary}
\begin{proof} Similarly to the previous proof, we consider another Lyapunov functional:
\[
 	\V_{\alpha}(\chi) = \int_0^1 \chi^{\top}(x) E_{\alpha}(x) \Lambda^{-1} S \chi(x) dx,
\]
and the extended state: $\xi = \left[ u_t(0) \ \ u_x(1) \ \ v_t(0) \ \ v_t(1) \right]^{\top}$.  $\V_{\alpha}$ is then equivalent to $\|\cdot\|_{\mathcal{H}}$ and it is exponentially stable in the sense of $\|\cdot\|_{\mathcal{H}}$. \end{proof}

The previous results are presented in terms of LMIs but for a stable wave and controller with $q = 0$, these inequalities are always verified as stated in the following corollary.
\begin{corollary} \label{sec:coroStable} If $h > 0, g > 0$ with $q = 0$, the unique solution of system \eqref{eq:system1}-\eqref{eq:controller1} with initial condition $(u^0, u_t^0, v^0, v_t^0) \in \mathcal{H}$ is $\mathcal{H}$-exponentially stable.
\end{corollary}
\begin{proof} First of all, note that $h > 0$ and $g > 0$ ensure that $\left| \frac{1 - c_1 g}{1 + c_1 g} \right| < 1$ and $\left| \frac{1 - c_2 h}{1 + c_2 h} \right| < 1$. Moreover by selecting $S_1 = S_2 = 0.5$ and $S_4 = S_3-\varepsilon$, we can rewrite $\tilde{\Psi}_{c_1, c_1g}$ in equation~\eqref{eq:PsiTilde} as:
\[
	\begin{array}{l}
		\tilde{\Psi}_{c_1, c_1g} = \diag \left(-2 c_1 g, \left[ \begin{smallmatrix} -c_2^2 \varepsilon &  c_1 - 2 c_2 S_3 + c_2 \varepsilon \\ c_1 - 2 c_2 S_3 + c_2 \varepsilon & -\varepsilon \end{smallmatrix} \right], \right. \ \ \ \ \\
		\hfill \left. \phantom{\left[ \begin{smallmatrix} S^2_3 \\ S^2_3 \end{smallmatrix} \right]} S_3 (c_2 h - 1)^2 - (S_3 - \varepsilon) (c_2 h + 1)^2 \right).
	\end{array}
\]

For $\varepsilon < \frac{c_1}{c_2}$ and $S_3 = \frac{c_1}{2 c_2} + \frac{\varepsilon}{2}$ then $\tilde{\Psi}_{c_1, c_1 g} \prec 0$ if and only if $S_3 (c_2 h - 1)^2 - (S_3 - \varepsilon) (c_2 h + 1)^2 < 0$. This is ensured by taking $\epsilon$ small enough. Consequently system \eqref{eq:system1}-\eqref{eq:controller1} is exponentially stable. \end{proof}

\begin{remark} Noting that for $c_1 = c_2 = c$ and $q = 0$, system \eqref{eq:system1}-\eqref{eq:controller1} is a wave equation of speed $c$ and length $2$. The dynamic controller acts then similarly than the one derived in \cite{gugat2015exponential}. The stability of this double boundary damped system is indeed $|\frac{1 - c g}{1 + c g} \frac{1 - c h}{1 + c h}| < 1$ as noted in \cite[ch. 3.3.1]{bastin2016stability}. The stability chart with respect to $cg$ and $ch$ is depicted in Figure~\ref{fig:stab}. Its decay-rate is given by the following formula:
\begin{equation} \label{eq:alphaDyn}
	\alpha_{dyn} = - \frac{c}{4} \log \left| \frac{1 - c g}{1+cg} \frac{1 - ch}{1+ch} \right|.
\end{equation}
\end{remark}

\begin{figure}
	\centering
	\includegraphics[width=9cm]{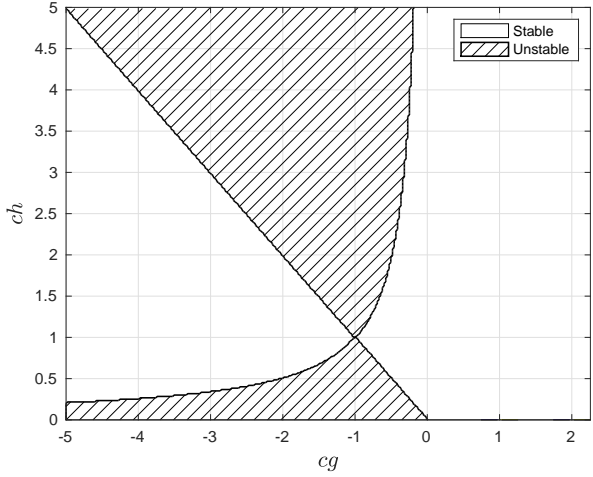}
	\caption{Stability areas for system \eqref{eq:system1}-\eqref{eq:controller1} depending on $c_1 = c_2 = c$, $h$ and $g$. The hatched area is unstable.}
	\label{fig:stab}
\end{figure}

The previous result can be seen as a robust stability criterion. Indeed, considering system~\eqref{eq:system1} with uncertain parameters $c_1 > 0$ and $g > 0$, the coupled system is stable no matter $c_2 > 0$ and $h > 0$.

However, in the case where $g$ is negative, the previous corollary does not apply. This is indeed more difficult because system~\eqref{eq:system1} is unstable and controller \eqref{eq:controller1} must be designed to make the interconnection stable. The next step is to derive a stability result for uncertain systems expressed in terms of LMIs and this is the aim of the following part.


\section{Robustness Analysis / Controller Synthesis}

Let us consider that system \eqref{eq:system1} is now an uncertain system, that is the speed $c_1$ and the damping $g$ are uncertain. Only a nominal system~\eqref{eq:system1} with nominal parameters $c=c_{n}$ and $g=g_{n}$ is known. We assume that the real speed $c_1$ is constant and that both $c_1$ and $c_n$ belong to the interval $[c_{min}, c_{max}] \subset \mathbb{R}^+$. The last parameter $g$ is assumed to be in the set $[g_{min}, g_{max}] \subset \mathbb{R}$ and so does $g_n$. Then, the uncertain system can be viewed as a deviation from the nominal plan.\\
If nominal plant~\eqref{eq:system1} with a speed $c_n$ and a damping coefficient $g_n$ is potentially unstable, then, a controller of the form of \eqref{eq:controller1} can be designed. Indeed, we set $c_2 = c_n$ and choose $h$ such that $(c_n g_n, c_n h)$ is in the stability area presented in Figure~\ref{fig:stab}. Then the nominal plant is exponentially stable with $q = 0$. The condition derived thereafter states the stability of the uncertain system made up of \eqref{eq:system1} and the previously designed controller.



\begin{theo} \label{sec:robustTheo} Let us define the following:
\[
	\begin{array}{l}
		\delta_{max} = \max_{z \in \mathcal{D}} \left| \frac{1-z}{1+z} \right|, 
		z_{max} = \Argmax_{z \in \mathcal{D}} \left| \frac{1-z}{1+z} \right|, \\
		\mathcal{D} = \left\{ z = c_1 g, c_1 \in [c_{min}, c_{max}], g \in [g_{min}, g_{max}] \right\}.
	\end{array}
\]

There exists a unique solution of system \eqref{eq:system1}-\eqref{eq:controller1} with initial condition $(u^0, u_t^0, v^0, v_t^0) \in \mathcal{H}$ and it is $\mathcal{H}$-exponentially stable if the following holds for $S_1, S_2, S_3, S_4 > 0$:
\[
	1 \leq \delta_{max} < \infty \quad \text{ and } \quad  \left\{ \begin{array}{l}
		\Psi_{c_{min}, z_{max}} \prec 0, \\
		\Psi_{c_{max}, z_{max}} \prec 0.
	\end{array} \right.
\]
\end{theo}

\begin{remark} Notice that coefficient $\delta(cg) = \frac{1 - cg}{1+cg}$ is a physical parameter for the wave equation and corresponds to the reflexion coefficient (see \cite{barreauInputOutput,louw2012forced}). Taking $\delta_{max} = \max_{cg \in \mathcal{D}} \left| \delta(cg) \right|$ means we are studying the ``most'' unstable system in the uncertainty set. 
For $\delta_{max} < 1$,  Corollary~\ref{sec:coroStable} states that the uncertain system is stable. If $\delta_{max} = + \infty$, there does not exist a controller of the form $\eqref{eq:controller1}$ making the system stable. These considerations bring that $\mathcal{D} \subset(-1, +\infty)$ or $\mathcal{D} \subset (-\infty, -1)$.\end{remark}

\begin{proof} The robustness analysis is based on the derivation of a common Lyapunov functional for all the systems inside the uncertainty set. This Lyapunov functional must have a strictly negative derivative along the trajectories. In other words, $\Psi_{c_1, z} \prec 0$ for all $c_1 \in [c_{min}, c_{max}]$ and $z \in \mathcal{D}$. 

Noticing that $\Psi_{c_1, z}$ is a block-diagonal matrix, one can write $\Psi_{c_1, z} = \diag(\Phi_z, \Theta_{c_1}, \Xi)$ with 
\[
	\begin{array}{l}
		\Phi_z = S_2 (z - 1)^2 - S_1(z + 1)^2, \\
		\Theta_{c_1} = \left[ 
			\begin{smallmatrix} 
				(S_1 - S_2 ) c_1^2 - (S_3 - S_4) c_2^2 & (S_1+ S_2) c_1 - (S_3 + S_4) c_2 \\
				(S_1 + S_2) c_1 - (S_3 + S_4) c_2 & S_1 - S_2 - S_3 + S_4
			\end{smallmatrix} \right],\\
		\Xi= \left[ \begin{smallmatrix} S_3(1-c_2h)^2-S_4(1+c_2h)^2 & -c_2q(S_3(1-c_2h)+S_4(1+c_2h)) \\
				-c_2q(S_3(1-c_2h)+S_4(1+c_2h)) & -S_4(c_2q)^2
			\end{smallmatrix} \right].
	\end{array}
\]

The aim is now to show that $\Phi_z < 0$, $\Theta_{c_1} \prec 0$ and $\Xi \prec 0$ for all uncertain systems. As $\delta_{max} < \infty$, then, $z \neq -1$ and $\Phi_z (z+1)^{-2} < \Phi_{z_{max}}(z_{max}+1)^{-2} < 0$. 

The last part of this proof deals with the negativity of $\Theta_{c_1}$. To derive such a result, one can prove that $\Theta_{c_1}$ is convex in $c_1$. Then if it is negative at its boundary, it is always negative. Using Schur complement, $\Theta_{c_1} \prec 0$ is equivalent to:
\begin{equation} \label{eq:c1rob}
	\left\{
	\begin{array}{l}
		p(c_1) = k_2 c_1^2 + k_1 c_1 + k_0 < 0, \\
		S_1 - S_2 - S_3 + S_4< 0,
	\end{array}
	\right.
\end{equation}
with $k_2 = S_1 - S_2 - \frac{(S_1 + S_2)^2}{S_1 - S_2 - S_3 + S_4}$ and $k_1, k_0 \in \mathbb{R}$.
Considering $\delta_{max} > 1$, we get: 
\[
	S_2  \leq S_2 \delta_{max}^2 < S_1,
\]
since $\Phi_{z_{max}} < 0$. Then, $k_2 > 0$ and consequently $p$ in \eqref{eq:c1rob} is convex with respect to $c_1$. Thus, if $\Theta_{c_{min}}$ and $\Theta_{c_{max}}$ are negative definite, the inequality $\Theta_{c_1} \prec 0$ straightforwardly holds for any $c_1$ in the interval $[c_{min},\ c_{max}]$. In other words, the following implication holds for all $c_1 \in [c_{min}, c_{max}]$ and $z \in \mathcal{D}$:
\[
	\left\{
		\begin{array}{l}
			\Psi_{c_{min}, z_{max}} \prec 0, \\
			\Psi_{c_{max}, z_{max}} \prec 0
		\end{array}
	\right.
	\Rightarrow
	\Psi_{c_1, z} \prec 0,
\]
which concludes the proof.
\end{proof}


\begin{remark} This proof shows that the result is still true even if $g$ is time-varying, but the restriction $g \in [g_{min}, g_{max}]$ holds. However, a time-varying $c_1$ is not allowed because it changes the calculations of the derivative along the trajectories of $\V$. \\
If one adds the constraint $q = 0$, the previous robustness result is still valid for $\tilde{\Psi}$ instead of $\Psi$.\end{remark}

\begin{remark} The proposed controller is not robust to delay in the measure so do the backstepping controllers in \cite{krstic2009delay}. \end{remark}

\section{Examples}

Two examples are studied in this part. The first one considers an open-loop stable  wave equation and the second one an unstable system \eqref{eq:system1}.

\subsection{Stable wave equation}

The first example is an open-loop  stable wave equation with a dynamic controller whose parameters are given below:
\[
	c_1 = c_2 = 1, \quad g = 3, \quad h = 0.9, \quad q = 5.
\]

The closed-loop system is stable according to Theorem~\ref{sec:theoUndamped}. Thanks to equation~\eqref{eq:alphaDyn}, the dynamic controller aims at making the system faster and converging to $0_\mathcal H$. The decay-rate of the solution can be obtained considering the maximum $\alpha$ in Theorem~\ref{sec:theoUndamped} for which it is still feasible. The resulting solution has then a decay-rate of $0.157$.
Notice that, since $q \neq 0$, the closed-loop system converges asymptotically to the only equilibrium point $0$, as shown  in Figure~\ref{fig:stable}. 

\begin{figure}
	\centering
	\includegraphics[width=9cm]{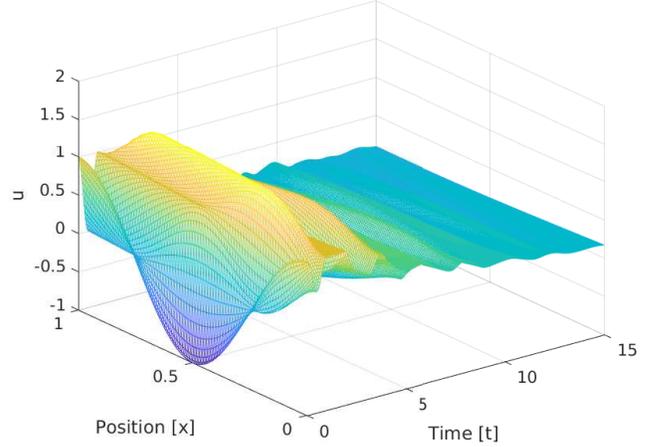}
	\caption{Time response of system \eqref{eq:system1} with a dynamic control and initial condition: $u^0(x) = \cos(2 \pi x)$, $v^0(x) = 1$ and $u^0_t(x) = v_t^0(x) = 0$.}
	\label{fig:stable}
\end{figure}

\subsection{Anti-damped wave equation}

This second example aims at comparing two controllers, the dynamic controller developed in this paper and the one coming from backstepping  approach \cite{krstic2009delay,smyshlyaev2009boundary}, which is  given by:
\begin{equation} \label{eq:controller2}
	w(t) = \frac{g - c_1 k}{c_1 - g k} \left( \frac{g}{c_1} u(0,t) + \int_0^1 u_t(x,t) dx \right).
\end{equation}

Applying this control law, system~\eqref{eq:system1} in closed loop is transformed into the following target system:
\[
	\left\{
	\begin{array}{ll}
		z_{tt}(x,t) = c_1^2 z_{xx}(x,t), \quad & x \in (0, 1), \\
		z_x(0,t) = k z_t(0,t), & \\
		z(1,t) = 0,
	\end{array}
	\right.
\]
where the initial conditions are not expressed, since it is the target system.
Its decay-rate is then given by $\alpha_{b} = -\frac{c_1}{2} \log \left| \frac{1 - c_1 k}{1 + c_1 k} \right|$. Comparing this expression to the decay rate of the proposed closed-loop target system \eqref{eq:alphaDyn}, to get a similar decay-rate (i.e. $\alpha_{dyn} = \alpha_b$) and then a fair comparison between the two controls, one must choose:
\[
	k = c_1^{-1} \frac{1 - \sqrt{|\delta_1 \delta_2|}}{1 + \sqrt{|\delta_1 \delta_2|}}, \quad \text{ or } \quad 
	k = c_1^{-1} \frac{1 + \sqrt{|\delta_1 \delta_2|}}{1 - \sqrt{|\delta_1 \delta_2|}}.
\]

\begin{figure*}
	\centering
	\subfloat[System \eqref{eq:system1} with controller \eqref{eq:controller1}]{\includegraphics[width=9cm]{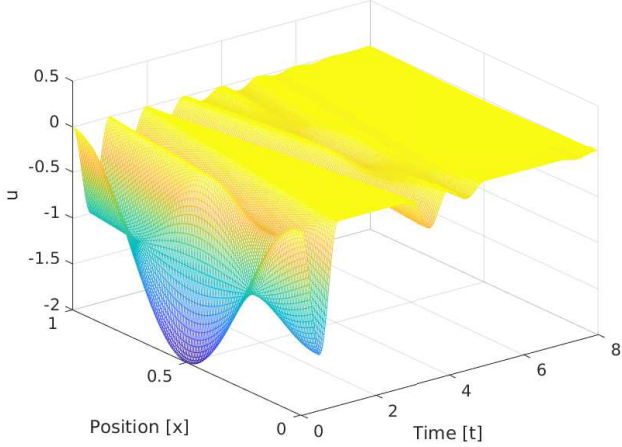}}
	\subfloat[System \eqref{eq:system1} with controller \eqref{eq:controller2}]{\includegraphics[width=9cm]{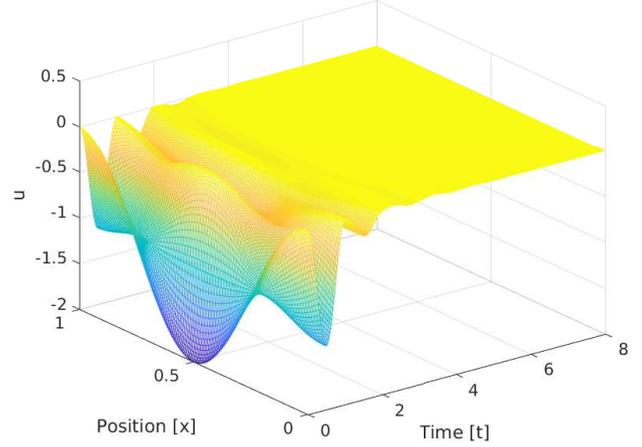}} \\
	\subfloat[Control signals $w$ based on the dynamic controller \eqref{eq:controller1} (blue) and \newline backstepping controller \eqref{eq:controller2} (dashed red).]{\includegraphics[width=9cm]{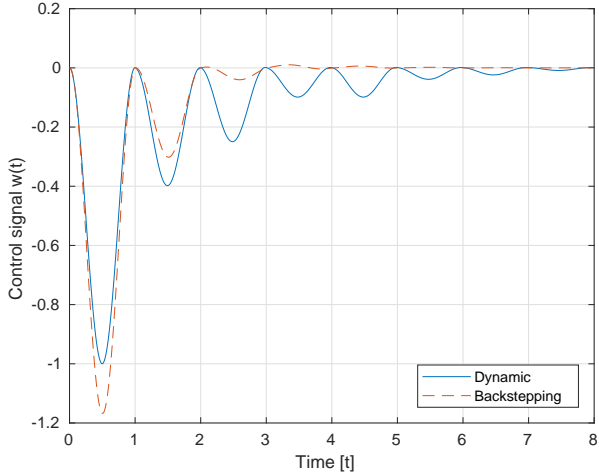} \label{fig:unstableControl}}
	\subfloat[Root locus of the closed-loop systems composed by \eqref{eq:system1} with dynamic controller \eqref{eq:controller1} (red) and backstepping controller \eqref{eq:controller2} (black)]{\includegraphics[width=9cm]{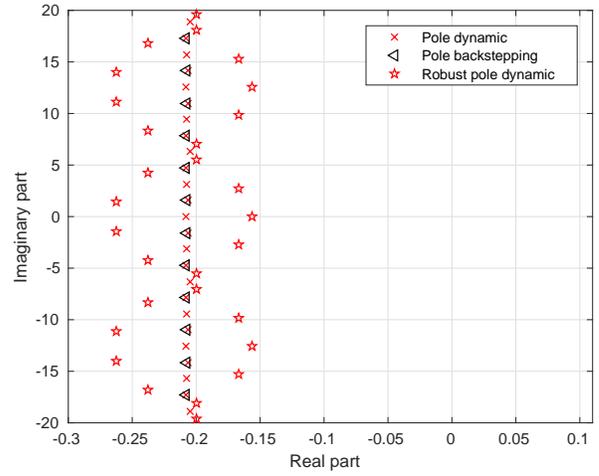} \label{fig:rlocusUnstable}}
	\caption{System behavior depending on the controller. Figures~(a) and (b) are simulations with initial conditions: $u^0(x) = \cos(2 \pi x) -1$, $u^0_t(x) = v^0(x) = v^0_t = 0$ for $x \in (0, 1)$. Figure~(c) is a comparison between the two control laws. Figure~(d) is a root locus comparison between the poles in three situations: backstepping \eqref{eq:controller2} with parameters \eqref{eq:simu1}; dynamic controller \eqref{eq:controller1} with parameters \eqref{eq:simu1} and robust controller \eqref{eq:controller1} with parameters \eqref{eq:simu1} but a speed mismatch ($c_1 = 0.8$).}
	\label{fig:unstable}
\end{figure*}


The parameters chosen for the simulation are then:
\begin{equation} \label{eq:simu1}
	c_1 = c_2 = 1, \ \ g = -0.27, \ \ h = 0.6, \ \ q = 0, \ \ k = 0.205,
\end{equation}
leading to a decay-rate of $0.208$. Figure~\ref{fig:unstable} shows the simulation results for system \eqref{eq:system1} for both controllers. The comparison between the two control signals is given in Figure~\ref{fig:unstableControl}.
The poles of each closed-loop system are displayed in Figure~\ref{fig:rlocusUnstable}. One can see that controller~\eqref{eq:controller1} produces more poles than the backstepping control. This is an argument showing that the two controls are indeed of different kind. The backstepping controller seems faster in Figure~\ref{fig:unstableControl} but Figure~\ref{fig:rlocusUnstable} clearly shows that the poles are indeed with the same real part, and consequently, with the same decay-rate. 
The backstepping controller seems faster in Figure~\ref{fig:unstableControl} but the controllers have been designed to have the same decay-rate.

One drawback of this methodology is that the dynamic controller does not provide an explicit control law while the backstepping, in this case, gives a relatively simple expression for $w$. Even if the backstepping methodology formulates $w$ in terms of $u$, it is also an infinite dimension control law since it uses the integral over space.

The main difference between the two controllers are the needed measurements. While the dynamic controller only requires the measure of $u_x(1,\cdot)$, the backstepping control asks for $u(0,\cdot)$ and $u_t(x, \cdot)$ for $x \in (0, 1)$. If these measurements are not available, an infinite-dimension observer has been developed in \cite{smyshlyaev2009boundary} to estimate the states $u_t$ at each point of the string but requiring the measurement of $u_x(1,\cdot)$ and $u_{xt}(1, \cdot)$. Then, the main advantage of having an explicit control law disappears. In comparison, our methodology provides a simpler control law with only one boundary measurement and a very simple robustness criterion. This is mainly due to the LMI formulation which provides an efficient framework for this kind of study. Moreover, the parameter $\gamma$ in Definition~\ref{def:expo} and equation~\eqref{eq:expEstimate} can be estimated using eigenvalues of $S$ in \eqref{eq:V1}.

A robustness analysis has also been conducted in this case using Theorem~\ref{sec:robustTheo}. We get the exponential stability for the interconnected system with $c_2 = 1, h = 0.6$ and the following uncertainties:
\begin{enumerate}
	\item $c_1 \in [0.74, 1.45]$ with $g \in [-0.27, + \infty)$;
	\item $c_1 \in [0.8, 1.4]$ with $g \in [-0.29, +\infty)$.
\end{enumerate}

There is no upper bound on $g$ as $\delta_{max}$ is always obtained for a negative $g$. According to the previous study, a mismatch between the two speeds ($c_1 = 0.8$ and $c_2 = 1$) leads to a stable interconnected system. In this case with $g = -0.27$, simulations confirm that point.




\section{Conclusion}

An infinite-dimensional controller is derived to stabilize a anti-stable string equation with Dirichlet actuation. 
It is quite simple and the calculations are easily extended to robust stability. This enables the comparison with backstepping and one can notice similar performances for a much more simple implementation. This study brought the idea of extending systems in order to transform them to a more suitable form for control. This idea can be enlarged to other PDEs and maybe also to cascaded PDEs.

\bibliographystyle{plain}
\bibliography{report_draft}

\end{document}